\documentclass[12pt]{amsart}
\usepackage[colorlinks=true,citecolor=green,linkcolor=magenta]{hyperref}
\usepackage{amsmath}
\usepackage{verbatim}
\usepackage{amsfonts}
\usepackage{amssymb}
\usepackage{blkarray}
\usepackage{color,colortbl}
\usepackage[all]{xy}  
\usepackage{enumerate}
\usepackage[top=1in, bottom=1in, left=1in, right=1in]{geometry}
\usepackage{mathrsfs}

\usepackage{stmaryrd}
\usepackage{bbold}
\renewcommand{\k}{\mathbb{k}}
\usepackage{cleveref}
\usepackage{soul}
\usepackage{arydshln,xcolor}
\usepackage{tikz-cd}
\usepackage{graphicx}
\theoremstyle{definition}
\newtheorem{theorem}{Theorem}[section]
\newtheorem{theoremx}{Theorem}

\numberwithin{equation}{section}

\newtheorem{corollary}[theorem]{Corollary}
\newtheorem{lemma}[theorem]{Lemma}
\newtheorem{proposition}[theorem]{Proposition}

\newtheorem{notation}[theorem]{Notation}

\theoremstyle{definition}
\newtheorem{definition}[theorem]{Definition}

\newtheorem{example}[theorem]{Example}

\newtheorem{conjecturex}[theoremx]{Conjecture}

\newtheorem{remark}[theorem]{Remark}

\newtheoremstyle{TheoremNum}
{8pt}{8pt}              
{\upshape}                      
{}                              
{\bfseries}                     
{.}                             
{.5em}                             
{\theoremname{#1}\theoremnote{ \bfseries #3}}
\theoremstyle{TheoremNum}

\newcommand{\m}{\mathfrak{m}}

\renewcommand{\)}{\right)}

\newcommand{\pp}{\partial}

\newcommand{\cC}{\mathfrak{C}}

\newcommand{\ord}{\operatorname{ord}}

\newcommand{\edim}{\operatorname{edim}}
\newcommand{\Rank}{\operatorname{rank}}

\newcommand{\C}{\mathfrak{C}}

\newcommand{\ms}{\operatorname{Mono}}

\renewcommand{\leq}{\leqslant}
\renewcommand{\geq}{\geqslant}
\newcommand{\ds}{\displaystyle}

\newcommand{\ps}[1]{\llbracket {#1} \rrbracket}

\renewcommand{\a}{\mathfrak{a}}

\newcommand\scalemath[2]{\scalebox{#1}{\mbox{\ensuremath{\displaystyle #2}}}}

\title{Valuations and nonzero torsion in module of differentials}

\address{Department of Mathematics, University of Utah, Salt Lake City, UT, USA}

\author[Maitra]{Sarasij Maitra}
\email{maitra@math.utah.edu}
\address{Department of Mathematics, Indian Institute of Technology Delhi, Hauz Khas, India.}

\author[Mukundan]{Vivek Mukundan}
\email{vmukunda@iitd.ac.in}

\begin{document}

	\begin{abstract}
Let $(R,\m_R,\k)$ be a one-dimensional complete local reduced $\k$-algebra over a field of characteristic zero. Berger conjectured that $R$ is regular if and only if the universally finite module of differentials $\Omega_R$ is torsion free. When $R$ is a domain, we prove the conjecture in several cases. Our techniques are primarily reliant on making use of the valuation semi-group of $R$. First, we establish a method of verifying the conjecture by analyzing the valuation semi-group of $R$ and orders of units of the integral closure of $R$. We also prove the conjecture in the case when certain monomials are missing from the monomial support of the defining ideal of $R$. These monomials are based on the smallest power of $\m_R$ that is contained within the conductor ideal. This also generalizes a previous result of \cite{ABC1}. 
	\end{abstract}
	\maketitle
\section{Introduction}
Let $X$ be a (affine or projective) curve over a field $\k$ of characteristic zero. The classification of circumstances when $X$ is smooth is of importance to both algebraists and geometers. Though there are other smoothness characterizations, the one involving the differential sheaf $\Omega_{X,x}$, where $x$ is a closed point on $X$, has eluded researchers. Consider the completion of the stalk of the structure sheaf $\mathcal{O}_X$ at $x$ and denote this by $\widehat{\mathcal{O}_{X,x}}$. R.W. Berger conjectured that $X$ is smooth at $x$ if and only if the sheaf of differentials $\widehat{\mathcal{O}_{X,x}}\otimes \Omega_{X,x}$ is torsion free. 

Let $R=\widehat{\mathcal{O}_{X,x}}\cong\k\ps{X_1,\dots,X_n}/I$ with $I\subseteq (X_1,\dots, X_n)^2$. For a curve, the differential sheaf mentioned above is $R\otimes \Omega_{X,x}$ where $\Omega_{X,x}$ is the module of K\"ahler differentials and $R\otimes \Omega_{X,x}$ is the universally finite module of differentials, henceforth denoted by $\Omega_R$. 
For the purposes of this article, we can define the universally finite module of differentials $\Omega_R$ to be the module whose presentation is as follows:
$$R^{\mu(I)}\stackrel{\mathcal{J}}{\to} R^n\to \Omega_{ R}\to 0$$ where $\mu(I)$ is the minimal set of generators of $I$ and $\mathcal{J}$ is the Jacobian matrix on a generating set for $I$, considered as a matrix over $R$. For further details, we refer the reader to the excellent resource \cite{Kunzbook}.
One of the first pieces of evidence on the smoothness of the curve $X$ via the torsion freeness of its differential sheaf was studied by E. Kunz (\cite{Kunz61}) when the ring is of characteristic $p$ and the base field is imperfect. Berger (\cite{Berger63}) in 1963  conjectured the following for perfect fields:
	\begin{conjecturex}[R. W. Berger \cite{Berger63}]\label{conjecture}
		Let $\k$ be a perfect field and let $R$ be a reduced one-dimensional analytic $\k$-algebra. Then $R$ is regular if and only if the torsion submodule $\ds \tau(\Omega_{ R})$ of $\Omega_R$ vanishes. 
	\end{conjecturex}
It follows from the above presentation that if $R$ is regular (i.e., $n=1$), then $\Omega_R$ is free.  Hence it is the converse direction that is of interest. Thus we assume $n\geq 2$ throughout this article.

 A direct generalization of this conjecture for higher dimensional surfaces does not hold (see for example \cite[Example Pg 160]{Matsuoka}, \cite[Satz 9.4]{scheja1970differentialmoduln}). But with added constraints on $R$ (such as Cohen-Macaulayness of $R$), a generalized version does exist due to R. H\"ubl and D. Arapura independently (\cite{Hubl,Arapura89}).  
There have been numerous instances in which \Cref{conjecture} has been established by a variety of different approaches (see \cite{scheja1970differentialmoduln}, \cite{herzog1971wertehalbgruppe}, \cite{Herzog78} \cite{Buchweitz}, \cite{Ulrich81} \cite{Koch},  \cite{HerzogWaldi84},  \cite{HerzogWaldi86}, \cite{Berger88}, \cite{pohl1989torsion1}, \cite{Guttes1990}, \cite{Isogawa}, \cite{Pohl91}, \cite{ABC1}, \cite{ABC2}, \cite{huneke2021torsion}, \cite{maitra2020partial}). 
An excellent exposition on the earlier ones among the above listed results can be found in \cite{Berger_article}. An Artinian variant of the conjecture, named the Artin-Berger conjecture, was explored in \cite{ABC1} recovering some of the earlier cases.


 
Recently, the authors of this article along with C. Huneke \cite{huneke2021torsion} demonstrated a new approach for explicitly constructing torsion in $\Omega_R$, where $R$ is assumed to be a domain with residue field $\k$ of characteristic $0$ and embedding dimension $n$. In that article, we provided a strategy for first constructing an over-ring $S$ such that $R\subseteq S\subseteq \overline{R}=\k\ps{t}$, and embedding dimension of $S$ is $n+s$.  Here $\overline{R}$ denotes the integral closure of $R$ in its field of fractions $\operatorname{Quot(R)}$. We then demonstrated that constructing ``sufficiently many" torsion elements in $\Omega_S$  enables us to pull back a torsion element to $\Omega_R$. In \Cref{torsions} of this article, we show that $$\lambda(\Omega_S)\geq ns+{s\choose 2}$$ (where $\lambda(M)$ denotes the length of a module $M$), which is one less than the required number of torsion elements meeting the above sufficiency criteria. This settles \cite[Remark 4.10]{huneke2021torsion}. 



Identifying $R$ within $\overline{R}=\k\ps{t}$ allows us to write $R=\k\ps{\alpha_1t^{a_1},\dots,\alpha_nt^{a_n}}$ where $\alpha_i$'s are units in $\overline{R}$. The valuation semi-group of $R$, namely $v(R)$, is critical for our methods throughout this article. 
In fact, the precise valuations $a_i$ and the valuation of the conductor ideal, $\cC_R=R:_{\text{Quot}(R)}\overline{R}$, are critical in our work of constructing torsion elements. Moreover, the ``order" of the unit $\alpha_d$ in the description of $x_d=\alpha_dt^{a_d}$ also plays a role. Given $\alpha_d=\sum_{i=0}^\infty u_{di}t^i$, a unit in $\k\ps{t}$, its order is defined to be $$o(\alpha_d)=v(\alpha_d-u_{d0})=\min_{i\geq 1}\{i~|~u_{ji}\neq 0 \}.$$
We state our first result here which is quite technical, but we illustrate with examples that the conditions are often satisfied (see \Cref{prop on two valuations} and \Cref{an+an-1 is bigger than c}).
\begin{theoremx}
	Let $R=\k\ps{\alpha_1t^{a_1},\dots,\alpha_nt^{a_n}}\subseteq \overline{R}=k[[t]]$ be as above with $\C_R=(t^{c_R})\overline{R}$ and $\k$,  an algebraically closed field of characteristic zero.
Then $\tau(\Omega_R)$ is nonzero in the following cases:
\begin{enumerate}
	\item the valuations $a_n,a_{n-1}$ satisfy $a_n+a_{n-1}\geq c_R$.
	\item 	
		 there exist $i,d$ such that $\ds o(\alpha_d)+a_i\geq c_R$, $\ds o(\alpha_d)<\infty$ and any one of the following statements holds.
		\begin{enumerate}
		    \item $o(\alpha_d)\leq o(\alpha_i)<\infty $,
		    \item $d\geq i$.
		\end{enumerate}
\end{enumerate}
\end{theoremx}


It is possible to get hold of a torsion element by mapping $\Omega_R$ to a rank one torsion free module (see \Cref{torsionfinding}). However, proving that it is non-zero is the most challenging part. Our discussion in  \Cref{monomial case} provides an approach towards achieving this goal, by going modulo a monomial ideal. We implement this method and obtain our second main result. This is the broadest class where we establish the truth of Berger's Conjecture (see \Cref{m^n in C}). 
\begin{theoremx}\label{second main}
		Let $R=\k\ps{X_1,\dots,X_n}/I$ be a one-dimensional local domain where $\k$ is algebraically closed of characteristic zero and $I\subseteq(X_1,\dots,X_n)^2$. Let $x_i$ denote the image of $X_i$ in $R$ and let the conductor be $\mathfrak{C}_R$. 
		 Suppose $\m^N\subseteq \mathfrak{C}_R$ and $x_1^{N-1},x_1^{N-2}x_2$, $N\geq 2$, are not in the monomial support ideal of $I$. Then the torsion submodule  $\tau(\Omega_R)$ is nonzero.
\end{theoremx}
The {monomial support ideal of $I$ is the same as $\ms(I)$ appearing in \mbox{
\cite{PoliniUlrichVitulli07}}\hspace{0pt}
. The } theorem recovers the main theorem of \cite{ABC1}. We also present other interesting examples that satisfy the hypothesis of the above theorem.

This raises the question if condition (2) in \Cref{second main} can be relaxed. So far our methods have been unable to resolve this. Getting a complete answer even in the case when $\cC_R=\m^N$ for some large enough $N$ would be interesting. The methods used in this article are innovative, in that, they explicitly compute torsion elements in $\Omega_R$, whereas most of the previous approaches to the \Cref{conjecture} have been unable to do so.

\subsection*{Acknowledgements}Both authors are extremely grateful to C. Huneke for insightful talks about the article's content. Numerous results in this article are ascribed to the concepts he imparted to us. Additionally, we are obliged to D. Arapura for providing us with a geometric viewpoint on Berger's Conjecture.
	\section{Setting and Preliminaries}\label{Section:Setting}
Throughout this article we will assume that $\k$ is an algebraically closed field of characteristic $0$, and $(R,\mathfrak{m}_R,\k)$ is a one-dimensional complete equicharacteristic local $\k$-algebra which is a domain with embedding dimension $n$ ($\mu_R(\mathfrak{m}_R)=n$).  Choosing $t$ to be a uniformizing parameter of the integral closure $\overline{R}$ of $R$,  we may assume that $\overline{R}=\k\ps{t}$. As a result, we may write $R=\k\ps{\alpha_1 t^{a_1},\dots,\alpha_nt^{a_n}}$ where $\alpha_i$'s are units in $\overline{R}$ and $a_1< a_2< \cdots< a_n$. Using the technique as in the first paragraph of the proof of \cite[Theorem 3.1]{huneke2021torsion}, we can alternatively write $R=\k\ps{ t^{a_1},\alpha_2t^{a_2},\dots,\alpha_nt^{a_n}}$. We briefly explain the process. First we multiply by a nonzero element of $\k$ to assume that the constant term of $\alpha_1$ is $1$. Let $s=\beta t$ where $\beta^{a_1}=\alpha_1$. Such a $\beta$ exists due to Hensel's Lemma \cite[THeorem 7.3]{eisenbud_Commalg}. The characterisitic of $\k$ being zero is important here. Notice that under this change of variables, $\k\ps{t}=\k\ps{s}$ and $s=\beta t=t+\beta_1 t^2+\beta_2t^3+\cdots$ where $\beta=1+\sum_{i=1}^\infty\beta_it^i$. Notice that $s^{a_1}=(\beta t)^{a_1}=\alpha_1t^{a_1}$. Also $\alpha_it^{a_i}=\alpha_i's^{a_i}$. Then $R=\k\ps{\alpha_1t^{a_1},\dots,\alpha_nt^{a_n}}\cong \k\ps{s^{a_1},\alpha_2's^{a_2},\dots,\alpha_n's^{a_n} }$.

We also define the valuation $v$ on $\overline{R}$ given by $v(p(t))=a$ if $p(t)=t^a\alpha$ where $\alpha$ is a unit in $\k\ps{t}$  (also called $\ord(-)$, see for example \cite[Example 6.7.5]{SwansonHuneke}). Let $v(R)$ denote the valuation semi-group of $R$.

 The conductor $\cC_R$ is the largest common ideal of $R$ and its
integral closure, $\overline R $. 
Since $\ds \overline{R}=k\ps{t}$, we have that $\cC_R=(t^i)_{i\geq c_R}\overline{R}$ where $c_R$ is the least integer such that $t^{c_R-1}\not \in R$, and $t^{c_R+i}\in R$ for all $i\geq 0$. The number $c_R$ is characterized as the least valuation in $\cC_R$. It is clear from this  discussion that there cannot be any element $r\in R$, such that $v(r)=c_R-1$.  Since $\ds \overline{R}$ is finitely generated over $R$ (\cite[Theorem 4.3.4]{SwansonHuneke}),  $\cC_R\neq 0$, and it is never all of $R$ unless $R$ is regular.

We define an epimorphism 
\begin{align*}
	\Phi:\k\ps{X_1,\dots,X_n}&\twoheadrightarrow R\\
	\Phi(X_i)&=\alpha_it^{a_i}\text{ for } 1\leq i\leq n.\nonumber 
\end{align*}
We denote the kernel of $\Phi$ by $I=(f_1,\dots,f_m)$ and hence have the natural isomorphism $R\cong \k\ps{X_1,\dots,X_n}/I$. Since $\edim R=n$, $I$ is contained in $\m_R^2$ where $\m_R=(X_1,\dots,X_n)$, the maximal ideal of $\k\ps{X_1,\dots,X_n}$.  Such rings are called \textit{analytic $\k$-algebras}. We will interchangeably use $\alpha_i t^{a_i}$ for $x_i$, the images of $X_i$ in the quotient $\k\ps{X_1,\dots,X_n }/I$.

\subsection{Universally Finite Module of Differentials}

\begin{definition}\label{defmoduleofdiff}
	Let $R$ be an analytic one-dimensional $\k$-algebra as above, which is a domain. Let $I=(f_1,\dots,f_m)$ where $f_j\in P=\k\ps{X_1,\dots,X_n}, n\geq 2$. We assume that $I\subseteq \m_P^2$ where $\m_P=(X_1,\dots,X_n)$. Then the \textit{universally finite module of differentials over $\k$}, denoted by $\Omega_{R}$, has a (minimal) presentation given as follows:
	$$R^{m}\xrightarrow{\left[\frac{\partial f_j}{\partial x_i}\right]} R^n\to \Omega_{ R}\to 0$$ where $\left[\frac{\partial f_j}{\partial x_i}\right]$ is the Jacobian matrix of $I$, with entries in $R$. 
\end{definition}
We refer the reader to the excellent resource \cite{Kunzbook} for more information. Let $\ds \tau(\Omega_{ R})$ denote the torsion submodule of $\Omega_{ R}$. 

%

%

\subsection{Computing Torsion}\label{torsionfinding} It is well-known that $\Rank_R(\Omega_R)=\Rank_R(\Omega_{\overline{R}})=1$ ($R$ and $\bar{R}$ have the same fraction field) and $\Omega_{\overline{R}}$ is torsion-free. Let $f:\Omega_{ R}\to \Omega_{\overline{R}}$ be the natural map induced by the inclusion $R\hookrightarrow\overline{R}$. Recall that the torsion submodule of $\Omega_R$ is the kernel of the homomorphism $\Omega_R\rightarrow \Omega_R\otimes Q$ where $Q=\text{Quot}(R)$. Since  $\Omega_{\overline{R}}$ is torsion-free, we have $\Omega_{\overline{R}}\hookrightarrow\Omega_R\otimes Q$. Thus $\ker f$ is the torsion submodule of $\Omega_R$ and hence $\tau(\Omega_R)=\ker f$. The map $f$ acts as follows: 
$$f(dx_i)=\frac{dx_i}{dt}dt.$$
Since $\Omega_{\overline{R}}$ is isomorphic to $\overline{R}$,  $\Omega_{R}$ surjects to an $R$-submodule $ \sum_{i=1}^nR\frac{dx_i}{dt}$ of $ \overline{R}=\k\ps{t}$. 
Thus the torsion submodule, $\tau(\Omega_R)$, consists of the column vectors $\begin{bmatrix}
	r_1\ldots r_n
\end{bmatrix}^t$ such that $ \sum_{i=1}^nr_i\frac{dx_i}{dt}=0$. Evidently, $\tau(\Omega_R)$ is nonzero precisely when the tuples $\begin{bmatrix}
	r_1\ldots r_n
\end{bmatrix}^t$ are not in the image of the presentation matrix (Jacobian matrix of $I$) of $\Omega_R$, all entries written in terms of the uniformizing parameter $t$. 

One of the main issues while constructing torsion is to check whether the column vectors $[r_1,\dots,r_n]^t$ is nonzero in $\Omega_R$. One technique is to show that this column vector is nonzero in $\Omega_{R/J}$ for some suitable ideal $J$ in $R$ which will then imply that it is nonzero in $\Omega_R$. Though we have a more general theorem for this purpose in \Cref{monomial case}, a version which we use in some of the results in the results in the sequel is the following theorem.
\begin{theorem}{\cite[Proposition 2.6]{ABC1}}
    Suppose $R=\k\ps{\alpha_1t^{a_1},\dots,\alpha_nt^{a_n}}$ with maximal ideal $\m_R=(x_1,\dots,x_n)$. Then the torsion $\tau(\Omega_{R/\m_R^2})$ has ${n\choose 2}$ $\k$-linearly independent torsion elements $\overline{x_i}d\overline{x_j}$ for $i<j$. Here $\overline{x_i}$ denotes the image of $x_i$ in $R/\m_R^2$.
\end{theorem}
We often avoid using the $\overline{\cdot}$ for notational convenience.
\subsection{Quasi-homogeneous rings} These rings form an important class which satisfies the Berger conjecture. A ring $R$ is said to be quasi-homogeneous, if there  is a surjection $\Omega_R\twoheadrightarrow \m_R$. Kunz and Ruppert in \cite{KunzRuppert77} showed that such rings are completions of graded rings. 
Scheja (\cite{scheja1970differentialmoduln}) showed that the torsion submodule of the module of differentials of such rings are always nonzero. In fact, he showed that $\mu(\tau(\Omega_R))\geq {\edim R\choose 2}$ when $R$ is quasi-homogeneous.
%
	\subsection{The extension $R[\frac{\mathfrak{C}_R}{x_1}]$}
The basic facts about the finite birational extension $S=R[\frac{\mathfrak{C}_R}{x_1}]$ of $R$ are  discussed in detail in \cite[Section 4]{huneke2021torsion}. 

We write $\mathfrak{C}_R/{x_1}$ to denote the set of elements of the form ${c}/{x_1}$ where $c\in \C_R$.  Using \cite[Proposition 2.9]{herzog1971wertehalbgruppe} and \cite[Lemma 4.1]{huneke2021torsion}, we get that $\ds \mathfrak{C}_R/x_1$ is generated as an $R$-submodule of $\overline{R}$ (in fact, it is the ideal $(t^{c-a_1})\overline{R}$) by $t^{b_i}$ with $b_1< \cdots< b_s$ where $b_i\in [c-a_1,c-1]\setminus v(R)$. For further details, see \cite[Notation 4.2]{huneke2021torsion}. It turns out that $\mathfrak{C}_R/x_1$ is the conductor of the ring $S$. In other words the conductor $\C_S$ equals the ideal $(t^{c-a_1})\overline{R}$. A complete description of the ring $S$ is as follows:

\begin{theorem}\cite[Thoerem 4.6]{huneke2021torsion}\label{kernelofnew_proof}
	Let  $\ds \mathfrak{C}_R\subseteq \m_R^2$. Set $P = \k\ps{X_1,...,X_n}$. Choose $S,b_1,\dots,b_s$ as described above. Then there exists a presentation of $S$ as follows: $$S=R\Big[\frac{\mathfrak{C}_R}{x_1}\Big]=\frac{\k\ps{X_1,\dots, X_n,T_1,\dots, T_s}}{\ker\Phi + \(X_iT_j-g_{ij}(X_1,\dots,X_n),T_kT_l-h_{kl}(X_1,\dots, X_n)\)_{{1\leq i\leq n, 1\leq j\leq s \atop 1\leq k\leq l\leq s}}}$$ where 
	$g_{ij}(X_1,\dots,X_n), h_{kl}(X_1,\dots, X_n)\in (X_1,\dots,X_n)^2P$ for all $i,j,k,l$. \\Moreover, $g_{ij}(x_1,\dots,x_n), h_{kl}(x_1,\dots, x_n)\in\C_R$. In particular, when $\C_R\subseteq \m_R^2$, it follows that $\edim S=n+s$. 

\end{theorem}

\section{Finer computations of Torsion elements in $\Omega_S$}\label{explicit Sec}
One of the main results in this section is the existence of torsion elements in $\Omega_S$. We show that we can in fact construct $ns+{s\choose 2}$ $\k$-linearly independent torsion elements in $\Omega_S$ which completes \cite[Remark 4.10]{huneke2021torsion}.

%
The following theorem establishes that once  $ns+{s\choose 2}+1$ torsion elements are constructed with entries in $\m_S$, then we can pull back a nonzero torsion element from $\Omega_S$ to $\Omega_R$. 
\begin{theorem}{\cite[Theorem 4.9]{huneke2021torsion}} \label{thm on one more torsion}
	Let $R=\k\ps{\alpha_1t^{a_1},\dots,\alpha_nt^{a_n}}$ with conductor $\mathfrak{C}_R\subseteq \m_R^2$. Construct $S=R\left[\frac{\mathfrak{C}_R}{x_1}\right]=R[T_1,\dots,T_s]$ as in \Cref{kernelofnew_proof}. Let $\tau(\Omega_R),\tau(\Omega_S)$ represent the torsion submodules of $\Omega_R,\Omega_S$ respectively. If $\lambda(\tau(\Omega_S))\geq ns+{s \choose 2}+1$ and all these torsion elements have non-units in the last $s$ rows (corresponding to $dT_1,\dots,dT_s$), then a $\k$-linear combination of these torsion elements can be pulled back to a nonzero torsion element in $\tau(\Omega_R)$. In particular, $\tau(\Omega_R)\neq 0$.
\end{theorem}
\begin{remark}\label{quick cons}An obvious consequence of this theorem is that if $\lambda(\tau(\Omega_S))\geq ns+{s \choose 2}+k, k\geq 1$ and all these torsions have non-units in the last $s$ rows, then we can pull back $k$ torsion elements to $\Omega_R$, each of which are $\k$-linear combinations of the $ns+{s \choose 2}+k$ torsion elements. Notice that these $k$ torsion elements are $\k$-linearly independent in $\Omega_R$ as they are $\k$-linearly independent in $\Omega_S$. Thus $\lambda(\tau(\Omega_R))\geq k$.
\end{remark}

{It is easy to see~} that we can construct ${s\choose 2}$ torsion elements, namely $\Gamma_{ij}:=b_jT_jdT_i-b_iT_idT_j$. This can be clearly seen as $T_i$ are monomials $t^{b_i}$ in $S$.  Now \Cref{torsions} shows that we can indeed construct $ns$ torsion elements in $\Omega_S$ too. 

If $x_i$ is a monomial $t^{a_i}$ to begin with, then we can construct the torsion element $a_ix_idT_j-b_jT_jdx_i$. This is clearly nonzero in $\Omega_S$. But often $x_i$ are not of this form and hence the need for the following theorem.

\begin{proposition}\label{torsions}
	Let $R=\k\ps{\alpha_1t^{a_1},\dots,\alpha_nt^{a_n}}$ with conductor $\mathfrak{C}_R$. Construct $S=R\left[\frac{\mathfrak{C}_R}{x_1}\right]=R[T_1,\dots,T_s]$. Then
	\begin{align}\label{torsion elements [x_i,T_j]}
		[x_i,T_j]=\begin{bmatrix}
			-a_ix_i\partial f_{ij}/\partial x_1\\
			\vdots\\
			b_j(T_j+f_{ij}(\underline{x})+\sum\limits_{k=1}^s\delta_k T_k)-a_ix_i\partial f_{ij}/\partial x_i\\
			\vdots\\
			-a_ix_i(1+\delta_j )\\
			\vdots\\
			-a_ix_i\delta_s
		\end{bmatrix}\hspace{-1em}
		\begin{array}{c;{2pt/2pt}c}
			~&~\\ ~&~\\ ~&~\\ ~&~\\ ~&~\\ ~&~\\ ~&~\\~&~\\
		\end{array}\hspace{-1em}
		\overbrace{
			\begin{array}{c}
				dx_1 \\ \vdots \\ dx_i \\ \vdots \\ dT_j \\\vdots \\ dT_s
			\end{array}
		}^{\text{basis}}
	\end{align}
	where $f_{ij}\in (x_1,\dots,x_n)$ and $\delta_r\in\k$, are $\k$-linearly independent nonzero torsion elements in $\tau(\Omega_S)$.


\end{proposition}
\begin{proof}
	We follow the process of ``monomialization'' as in \cite[Theorem 3.1]{huneke2021torsion} (see \Cref{Section:Setting}). Let $\beta=\alpha_i^{1/a_i}$ and $s=\beta t$ (since the characteristic of $\k$ is $0$, by Hensel's lemma \cite[Theorem 7.3]{eisenbud_Commalg}, there exists such a $\beta\in \overline{R}$). Thus $R\cong\k\ps{\alpha_1's^{a_1},\alpha_2's^{a_2},\dots,s^{a_i},\dots,\alpha_n's^{a_n}}$. 
	Under this isomorphism (change of variables), let the new variables be $x_1',\dots,x_n',T_1',\dots,T_s'$. The valuations do not change under this isomorphism. Since $T_1'\dots,T_s'\in\C_S$, we assume that, without loss of generality, under another change of variables, $T_1',\dots,T_s'$ are monomials $s^{b_1},\dots,s^{b_s}$. Notice that under this change of variables $x_i'=s^{a_i}$. Clearly, for $j=1,\dots,s$ we have 
	\begin{align}\label{equation of torsion in new vars}
	b_jT_j' d(x_i')-a_ix_i'd(T_j')
	\end{align}
	 is a nonzero torsion element of $\Omega_S$ (the nonzero part can be obtained by  {showing that $\overline{b_jT_j' d(x_i')-a_ix_i'd(T_j')}\equiv (b_j-a_i)\overline{T_j'}d\overline{x_i'}$ in $\Omega_{S/\m_S^2}$, which is nonzero} using \cite[Proposition 2.6]{ABC1}). 

	Now we convert it back to the original variables $x_1,\dots,x_n,T_1,\dots,T_s$. Under the change of variables, $x_i$ is mapped to $x_i'$. So we just need to convert $T_j'$ back to $T_j$. We have $T_j'=s^{b_j}=\beta^{b_j}t^{b_j}=t^{b_j}+\sum _k \gamma_kt^{b_j+k}$. Since  {$\mathfrak{C}_S=(t^{c_R-a_1})\overline{R}$ and $c_R-a_1$} is smaller than $b_j$, we can replace $t^{b_j+k}$ as a function of $x_1,\dots,x_n,T_1,\dots,T_s$. Thus we have
	\begin{align*}
		T_j'=t^{b_j}+\sum _k \gamma_kt^{b_j+k}=T_j+g_{ij}(x_1,\dots,x_n,T_1,\dots,T_s).
	\end{align*}
	Due to the nature of the defining ideal of $S$,
	\begin{align}\label{f_i's in [x_i,T_j]}
		g_{ij}(x_1,\dots,x_n,T_1,\dots,T_s)=f_{ij}(x_1,\dots,x_n)+\sum\limits_{k>j}^s\delta_{ik} T_k	
	\end{align}
	where $\delta_{ik}\in\k$. In the latter equality, only $T_k$ with $k>j$ appear because the only nonzero terms which can appear in $g_{ij}(\underline{x},\underline{T})$ are those which have valuations more than that of $T_j$.

	Therefore
	\begin{align*}
		dT_j'=dT_j+\sum_{u=1}^n\frac{\partial f_{ij}}{\partial x_u}dx_u+\sum_{k>j}^s\delta_{ik} dT_k
	\end{align*}
	Putting the above equation in \eqref{equation of torsion in new vars} we have
	\begin{align}
		[x_i,T_j]=\begin{bmatrix}
			-a_ix_i\partial f_{ij}/\partial x_1\\
			\vdots\\
			b_j(T_j+f_{ij}(\underline{x})+\sum\limits_{k>j}^s\delta_{ik} T_k)-a_ix_i\partial f_{ij}/\partial x_i\\
			\vdots\\
			-a_ix_i\\
			-a_ix_i\delta_{ij+1}\\
			\vdots\\
			-a_ix_i\delta_{is}
		\end{bmatrix}\hspace{-1em}
 \begin{array}{c;{2pt/2pt}c}
			~&~\\ ~&~\\ ~&~\\ ~&~\\ ~&~\\ ~&~\\ ~&~\\~&~\\~&~\\~&~\\
		\end{array} \hspace{-1em}
		\overbrace{
			\begin{array}{c}
				dx_1 \\ \vdots \\ dx_i \\ \vdots \\ dT_j\\dT_{j+1} \\\vdots \\ dT_s 
			\end{array}
		}^{\text{basis}}.
	\end{align}

	Next we show that $[x_i,T_j], 1\leq i\leq n, 1\leq j\leq s$ are $\k$-linearly independent. First let $f_{ij}(x_1,\dots,x_n)=\sum_k\beta_{ijk}x_k+f_{ij}'(x_1,\dots,x_n)$ where $f_{ij}'(x_1,\dots,x_n)\in\m_R^2,\beta_{ijk}\in\k$. Now consider the torsion elements $[x_i,T_j]$ in $\Omega_{S/\m_S^2}$. These would take the form
	\begin{align}\label{description of [x_i,T_j] in S/m2}
		\begin{bmatrix}
			-a_ix_i\beta_{ij1}\\\vdots \\-a_ix_i\beta_{iji-1}\\b_j(T_j+\sum_{k\neq i}\beta_{ijk}x_k+\sum\limits_{k>j}^s\delta_{ijk} T_k)\\-a_ix_i\beta_{iji+1}\\\vdots\\-a_ix_i\beta_{ijn}\\0\\\vdots\\0 \\-a_ix_i\\\vdots\\-a_ix_i\delta_{ijs}
		\end{bmatrix}&= \begin{bmatrix}
			0\\\vdots \\0\\(b_j-a_i)T_j+\sum\limits_{k>j}^s\delta_{ijk}(b_j-a_i) T_k-(b_j+a_i)\sum_{k\leq i}\beta_{ijk}x_k\\-(b_j+a_i)x_i\beta_{iji+1}\\\vdots\\-(b_j+a_i)x_i\beta_{ijn}\\0\\\vdots\\0
		\end{bmatrix},
	\end{align}
	where the equality is due to the fact that  {$x_1x_i,\dots,x_{i-1}x_i,x_iT_1,\dots,x_iT_s\in\m_S^2$ }(the latter   {containment implies that $x_jdx_i=-x_idx_j,j\leq i$ } and $x_idT_j=-T_jdx_i$ in $\Omega_{S/\m_S^2}$). So to show that $[x_i,T_j], 1\leq i\leq n, 1\leq j\leq s$ are $\k$-linearly independent, it is enough to show that these torsions are $\k$-linearly independent in $\Omega_{S/\m_S^2}$.
	
	We first show that $[x_i,T_j], 1\leq j\leq s$, for a fixed $i$, are $\k$-linearly independent. It is known that $x_kdx_l,T_jdx_i,1\leq k<l\leq n,1\leq j\leq s$ are $\k$-linearly independent  in $\Omega_{S/\m_S^2}$ 
\cite[Proposition 2.6, Corollary 2.7]{ABC1}. Now writing the coefficients of part of the ordered basis $x_kdx_l,T_jdx_i,1\leq k<l\leq n,1\leq j\leq s$ in \eqref{description of [x_i,T_j] in S/m2}, we get a $s\times ({n\choose 2}+ns)$ matrix
	 \begin{align*}
 \scalemath{0.62}{\begin{bmatrix}
[x_i,T_1]\\ [x_i,T_2]\\ [x_i,T_3]\\ \vdots \\ [x_i,T_s]
\end{bmatrix}=
		\left[\begin{array}{ccccccccc|cccccccccc}
			0 &\cdots &0 &-(b_1+a_i)\beta_{i11} & \cdots &-(b_1+a_i)\beta_{i1n}&0 &\cdots &0&0 &\cdots &0&	b_1-a_i & \delta_{i12}(b_1-a_i) &\cdots & \delta_{i1s} (b_1-a_i)&0 &\cdots &0\\
			0 &\cdots &0 &-(b_2+a_i)\beta_{i21} & \cdots &-(b_1+a_i)\beta_{i2n}&0 &\cdots &0&0 &\cdots &0&	0 & b_2-a_i & \cdots & \delta_{i2s}(b_2-a_i)&0 &\cdots &0\\
		0 &\cdots &0 &0 & \cdots &0&0 &\cdots &0&0 &\cdots &0&	0 & 0 & \cdots& \delta_{i3s}(b_3-a_i)&0 &\cdots &0\\
		\vdots &\vdots &\vdots&\vdots &\vdots &\vdots &\vdots&\vdots &\vdots&\vdots &\vdots &\vdots&	\vdots &\vdots & \ddots &\vdots&\vdots &\vdots &\vdots\\
			0 &\cdots &0 &-(b_s+a_i)\beta_{is1} & \cdots &-(b_1+a_i)\beta_{isn}&0 &\cdots &0&0 &\cdots &0&	0 & 0 & 0\cdots 0 &b_s-a_i&0 &\cdots &0\\
		\end{array}\right]\left[\begin{array}{c}
		x_1dx_2\\\vdots\\x_{n-1}dx_n\\\hline T_1dx_1\\\vdots\\ T_sdx_n
	\end{array}\right] }
	 \end{align*}
Assume the first block is $A_i$ which is a $s\times {n\choose 2}$  matrix, and the second block is $[0\cdots 0 ~B_i~0\cdots 0]$ which is a $s\times {ns}$ matrix, where $B_i$ is an upper triangular $s\times s$ matrix. The columns of $A_i$ correspond to the ordered basis $x_kdx_l,1\leq k<l\leq n$ and that of $[0\cdots 0 ~B_i~0\cdots 0]$ correspond to the ordered basis $T_jdx_i,1\leq j\leq s$. Clearly the matrix $B_i$ has rank $s$ as $a_i,b_1,\dots,b_s$ are distinct and nonzero. Since the rows are $\k$-linearly independent,  \\ {$[x_i,T_j]=\left((b_j-a_i)T_j+\sum\limits_{k>j}^s\delta_{ijk}(b_j-a_i) T_k-(b_j+a_i)\sum_{k\leq i}\beta_{ijk}x_k\right)dx_i-\sum_{k=i+1}^na_ix_i\beta_{ijk}dx_k$ } are $\k$-linearly independent for a fixed $i$  {and $1\leq j\leq s$}.

If we consider all the torsion elements $[x_i,T_j]$ for arbitrary $i,j$, then we get the following matrix
\begin{align*}
\begin{bmatrix}
	[x_1,T_1]\\\vdots\\ [x_n,T_s]
\end{bmatrix}=\begin{bmatrix}
	A_1 & B_1 & 0 &0 &\cdots & 0\\
	A_2 & 0 & B_2&0 &\cdots & 0\\
	\vdots &\vdots & \vdots &\vdots &\vdots &\vdots\\
	A_n & 0 & 0 & 0&\cdots & B_n
\end{bmatrix} \left[\begin{array}{c}
x_1dx_2\\\vdots\\x_{n-1}dx_n\\\hline T_1dx_1\\\vdots\\ T_sdx_n
\end{array}\right]
 \end{align*}
where each of the $B_i$ are upper triangular invertible $s\times s$ matrix. Thus the rows are $\k$-linearly independent and hence $[x_i,T_j]$'s are $\k$-linearly independent.
\end{proof}
\begin{theorem}\label{totalcount}
	With the hypothesis as in \Cref{torsions}, we have \[\lambda(\tau(\Omega_S))\geq ns+{s\choose 2}.\]
\end{theorem}

\begin{proof}We observe that $\Gamma_{ij}=b_jT_jdT_i-b_iT_idT_j$ and $[x_i,T_j]$ are $\k$-linearly independent: suppose $$0=\sum_{i,j}k_{ij}[x_i,T_j]+\sum_{r,s} k'_{rs}\Gamma_{rs}, k_{ij},k'_{rs}\in\k.$$ Now let $J=\langle x_1,\dots,x_n\rangle +\langle T_1,\dots,T_s\rangle^2$. Then in $\Omega_{S/J}$, $[x_i,T_j]=0$ (using \eqref{description of [x_i,T_j] in S/m2}). Thus in $\Omega_{S/J}$, we have $0=\sum_{r,s}k'_{rs}\Gamma_{rs}$. Since $\Gamma_{rs}$ are  $\k$-linearly independent (\cite[Proposition 2.6]{ABC1}), we have $k'_{rs}=0$. Thus we have $0=\sum_{i,j}k_{ij}[x_i,T_j]$ in $\Omega_S.$ Hence $k_{ij}=0$ by \Cref{torsions}. Thus $\lambda(\tau(\Omega_S))\geq ns+{s\choose 2}$.\end{proof}


 \begin{remark}\label{quasihomogeneous}
	In \cite[Theorem 4.11]{huneke2021torsion}, we established that $\tau(\Omega_R)$ is nonzero if $S=R\left[\frac{\C_R}{x_1}\right]$ is quasi-homogeneous. Recall that quasi-homogeneous means that there is a surjection $\Omega_{ R}\twoheadrightarrow \m_R$. In fact, the proof mentioned above proceeded by showing that there are ${n+s\choose 2}$ $\k$-linearly independent torsion elements in $\Omega_S$ and none of these have units in the last $s$ rows. However, notice that ${n+s\choose 2}=ns+{s\choose 2}+\binom{n}{2}$. Thus we can pull back ${n\choose 2}$ torsions to $\Omega_R$ using \Cref{thm on one more torsion} and \Cref{quick cons}, i.e., $\lambda(\tau(\Omega_R))\geq \binom{n}{2}$.




%
\end{remark}

	\section{Value semi-group of $R$ and nonzero torsion}
	In this section, we provide partial answers to Berger's conjecture by carefully observing the value semi-group of $R$ as well as order of the units $\alpha_j$ in $\k\ps{t}$.

We begin by stating a useful test of checking the nonvanishing of torsion in $\Omega_R$.  First notice that units in $\k\ps{t}$ are power series of the form $\sum_{i=0}^\infty u_it^i$ where $u_0\neq 0$.
\begin{notation}\label{notation of units}
	Let $\alpha_j=\sum_{i=0}^\infty u_{ji}t^i\in \k\ps{t}$ be a unit. We denote 
	\begin{align*}
		o(\alpha_j)=v(\alpha_j-u_{j0})=\min_{i\geq 1}\{i~|~u_{ji}\neq 0 \},
	\end{align*}
where $v(\cdot)$ denotes the order valuation (also called $\ord(-)$, see for example \cite[Example 6.7.5]{SwansonHuneke}). If $\alpha_j=u_{j0}\in\k$, then notice that $o(\alpha_j)=\infty$.
\end{notation}

	\begin{theorem}\label{simultaneously diagonalizable}
		Let the conductor of $R$ be $\mathfrak{C}_R=(t^{c_R})\overline{R}$. 
		Suppose there exists $i,d$ such that $\ds o(\alpha_d)+a_i\geq c_R$ and $\ds o(\alpha_d)<\infty$. Then the torsion $\tau(\Omega_R)$ is nonzero if any one of the following statements holds.
		\begin{enumerate}
		    \item $o(\alpha_d)\leq o(\alpha_i)<\infty $,
		    \item $d\geq i$.
		\end{enumerate}
		 

	\end{theorem}

	\begin{proof}As in \cite[Theorem 3.1]{huneke2021torsion}, we write  $R= \k\ps{t^{a_1},\alpha_2t^{a_2}, \dots, \alpha_nt^{a_n}}$ with conductor $\C_R=(t^{c_R})\overline{R}$. 
	Notice that the $o(\alpha_i)$ remain the same by applying the above change of variables. For $(1)$, observe that $o(\alpha_i)+a_i\geq o(\alpha_d)+a_i\geq c_R$. Now 
	\begin{align*}
	    x_i=\alpha_it^{a_i}=\alpha_{i0}t^{a_i}+t^{o(\alpha_i)+a_i}b
	\end{align*} where $b\in \overline{R}$. Since $o(\alpha_i)+a_i\geq c_R$, $t^{o(\alpha_i)+a_i}b\in \cC_R\subseteq R$. Thus, $t^{a_i}=\frac{1}{\alpha_{i0}}(x_i-t^{o(\alpha_i)+a_i}b)\in R$. Thus, we get that $R=k\ps{t^{a_1},\alpha_2t^{a_2},\dots,t^{a_i},\dots,\alpha_n t^{a_n}}$. The proof is now complete using \cite[Remark 3.3]{huneke2021torsion}.
	
	For $(2)$, first observe that by our setup, $o(\alpha_d)+a_d\geq o(\alpha_d)+a_i\geq c_R$. Apply the same argument as above to see that $R=k\ps{t^{a_1},\alpha_2t^{a_2},\dots,t^{a_d},\dots,\alpha_n t^{a_n}}$. This finishes the proof.

	\end{proof}
	
The following is an example of a non-quasi-homogeneous ring $R$ where the above result can be quickly applied.
	\begin{example}
		Let $R=\k\ps{t^{8}+t^{9},t^{9}+t^{15},t^{12}+t^{20},t^{14}}=\k\ps{t^8(1+t),t^9(1+t^6),t^{12}(1+t^8),t^{14}}$. Macaulay2 computations show that the conductor is $(t^{c_R})\overline{R}=(t^{20})\overline{R}$. With $i=3, d=3$, we see that $a_i+o(\alpha_d)=12+8=20=c_R$. Thus by the above theorem, the torsion $\tau(\Omega_R)$ is nonzero. Moreover, Macaulay 2 computations show that $v(\text{trace}(\Omega_{ R}))=16$ where $\text{trace}(\cdot)$ refers to the trace ideal (see \cite{LindoTrace} for details). Hence, $R$ is not quasi-homogeneous \cite[Conclusion 2]{maitra2020partial}.
	\end{example}
	The second part of the next result appears as \cite[Theorem 4.9]{phdthesis}.
	\begin{theorem}\label{prop on two valuations}
		Suppose the valuations $a_1,\dots,a_n$ of $x_1,\dots,x_n$ of $R$ satisfies one of the following:
		\begin{enumerate}
			\item $a_1+a_n\geq c_R$,
			\item $a_{n-1}+a_n\geq c_R+a_1$
		\end{enumerate}
		Then the torsion $\tau(\Omega_R)$ is nonzero.
	\end{theorem}
	\begin{proof} Throughout the proof, we can assume that $\mathfrak{C}_R\subseteq \m_R^2$ as otherwise the conclusion follows from \cite[Theorem 3.1]{huneke2021torsion}. Further applying the same ``monomialization" as in the proof of \Cref{simultaneously diagonalizable}, we can assume that $R=k\ps{t^{a_1},\alpha_2t^{a_2},\dots,\alpha_nt^{a_n}}$.
		
		(1): Consider $S=R\left[\frac{\mathfrak{C}_R}{x_1}\right] $ with conductor $\mathfrak{C}_S=(t^{c-a_1})\overline{R}$. Since $a_n\geq c-a_1$, $x_n$, by change of variables, can be made into  monomial $t^{a_n}$. Thus $t^{a_n}=x_n'=x_n-\sum \delta_i T_i-g(\underline{x}), \delta_i\in\k$. Then $S=\k\ps{x_1,\dots,x_{n-1},x_n',T_1,\dots,T_s}=\k\ps{t^{a_1},\alpha_2t^{a_2},\dots,\alpha_{n-1}t^{a_{n-1}},t^{a_n},t^{b_1},\dots,t^{b_s}}$. Clearly $\tau=a_nx_n'dx_1-a_1x_1dx_n'$ is a nonzero torsion in $\Omega_S$. Hence (using the notations as in \Cref{explicit Sec}) {$\tau,[x_i,T_j],\Gamma_{ij}$ } are $ns+{s\choose 2}+1$ torsion elements   {in $\Omega_S$. If we show that these torsion elements are $\k$-linearly independent, then the result follows from} \Cref{thm on one more torsion}{.
		}

		{Consider 
		}\begin{align}{\label{valuations eq1}
			\gamma \tau+\sum k_{ij}[x_i,T_j]+\sum k_{ij}'\Gamma_{ij}=0.
		}\end{align}
		{First let $J=\langle x_1,\dots,x_n'\rangle+\m^2_S$; then in $\Omega_{S/J}$, }\eqref{valuations eq1} {takes the form $\sum k'_{ij}\Gamma_{ij}=0$. But these are $\k$-linearly independent and hence $k_{ij}'=0$. Now let $J_1=\langle x_n'\rangle+\m_S^2$. Then in $\Omega_{S/J_1}$, $\tau=0$ as $x_n'\in J_1$ and $dx_n'=0$ in $\Omega_{S/J_1}$. Thus }\eqref{valuations eq1} {now becomes $\sum_{i\neq n} k_{ij} [x_i,T_j]=0$. Since we already proved $[x_i,T_j]$'s } are $\k$-linearly independent{, we have $k_{ij}=0$ for $i\neq n, 1\leq j\leq s$. Now }\eqref{valuations eq1} {becomes $\gamma\tau+\sum_jk_{nj}[x_n',T_j]=0$. Now let $J_2=\langle x_1\rangle+\m^2_S$. Then in $\Omega_{S/J_2}$, we have again $\tau=0$, but $0\neq [x_n',T_j]=a_nx_n'dT_j-b_jT_jdx_n'=(a_n+b_j)x_n'dT_j$. Thus  we have $k_{nj}=0$ as well. We now have $\gamma\tau=0$ which implies $\gamma =0$. Thus $\tau,[x_i,T_j],\Gamma_{ij}$ are $\k$-linearly independent}.

		%
		%

		(2):  
		In $\Omega_{\overline{R}}\cong \k\ps{t}$ where $\overline{R}=\k\ps{t}$, we have 
		\begin{align}
			\frac{dx_n}{dt}\left(\frac{dx_1}{dt}\right)&-\frac{dx_1}{dt}\left(\frac{dx_n}{dt}\right)=0\nonumber\\
			\left(a_n\alpha_nt^{a_n-1}+t^{a_n}\frac{d(\alpha_n)}{dt}\right)\left(\frac{dx_1}{dt}\right)&-a_1t^{a_1-1}\left(\frac{dx_n}{dt}\right)=0\label{section 6 eq1}
		\end{align}
		Multiplying \eqref{section 6 eq1} by $\alpha_{n-1}t^{a_{n-1}-a_1+1}$, we get
		\begin{align}
			\left(a_n\alpha_n\alpha_{n-1}t^{a_n-1+a_{n-1}-a_1+1}+\alpha_{n-1}t^{a_n+a_{n-1}-a_1+1}\frac{d\alpha_n}{dt}\right)\frac{dx_1}{dt}&-a_1\alpha_{n-1}t^{a_1-1+a_{n-1}-a_1+1}\frac{dx_n}{dt}=0\nonumber\\
			\left(a_n\alpha_n\alpha_{n-1}t^{a_n+a_{n-1}-a_1}+\alpha_{n-1}t^{a_n+a_{n-1}-a_1+1}\frac{d\alpha_n}{dt}\right)\frac{dx_1}{dt}&-a_1\alpha_{n-1}t^{a_{n-1}}\frac{dx_n}{dt}=0\nonumber\\
			\left(a_n\alpha_n\alpha_{n-1}t^{a_n+a_{n-1}-a_1}+\alpha_{n-1}t^{a_n+a_{n-1}-a_1+1}\frac{d\alpha_n}{dt}\right)\frac{dx_1}{dt}&-a_1x_{n-1}\frac{dx_n}{dt}=0\label{section 6 eq2}
		\end{align}
		The order valuation of $u=a_n\alpha_n\alpha_{n-1}t^{a_n+a_{n-1}-a_1}+\alpha_{n-1}t^{a_n+a_{n-1}-a_1+1}\frac{d\alpha_n}{dt}$  is at least $a_n+a_{n-1}-a_1$. Since $a_n+a_{n-1}-a_1\geq c_R$, we get that $u\in \cC_R\subseteq R$. Thus \eqref{section 6 eq2}, as a member of $\Omega_R$, takes the form 
		\begin{align}\label{section 6 eq3}
			\begin{bmatrix}
				u\\0\\\vdots\\0\\a_1x_{n-1}
			\end{bmatrix}
		\end{align}
	 Since $\mathfrak{C}_R\subseteq\m_R^2$, consider the torsion element \eqref{section 6 eq3} in $\Omega_{R/\m_R^2}$. It takes the from $a_1x_{n-1}d(x_n)$ which is nonzero due to \cite[Proposition 2.6, Corollary 2.7]{ABC1}.
	\end{proof}
\begin{remark}\label{description of [x_i,T_j] with an bigger than c-a1}
	An immediate consequence of the above theorem is that, henceforth, we can always assume $a_n<c_R-a_1$. This in turn results in the simplification of the description of $[x_i,T_j]$ \eqref{description of [x_i,T_j] in S/m2} in $\Omega_{S/\m_S^2}$. First, consider the description of $f_{ij}(x_1,\dots,x_n)$ in \eqref{f_i's in [x_i,T_j]}. The valuations of the terms in $f_{ij}(x_1,\dots,x_n)$ are more than that of $T_j$, or more than $c_R-a_1$. Since $a_n<c_R-a_1$, no linear terms can appear in $f_{ij}(x_1,\dots,x_n)$. Thus $f_{ij}(x_1,\dots,x_n)\in\m_R^2$ and hence $\beta_{ijk}=0$. Thus  \eqref{description of [x_i,T_j] in S/m2} takes the form 
	\begin{align}
\begin{bmatrix}\label{simpler description of [x_i,T_j]}
			0\\\vdots \\0\\(b_j-a_i)T_j+\sum\limits_{k>j}^s\delta_{ik}(b_j-a_i) T_k\\0\\\vdots\\0
		\end{bmatrix}\hspace{-1em}
  \begin{array}{c;{2pt/2pt}c}
	~&~\\ ~&~\\ ~&~\\ ~&~\\ ~&~\\ ~&~\\ ~&~\\ ~&~\\ ~&~\\ ~&~
\end{array}\hspace{-1em}
\overbrace{
\begin{array}{c}
	dx_1 \\ \vdots \\ dx_{i-1}\\dx_i\\dx_{i+1} \\ \vdots \\ dT_s 
\end{array}
}^{\text{basis}}
	\end{align}
in $\Omega_{S/\m_S^2}$. This simpler form will help us to prove $\k$-linear independence of the torsion elements.
\end{remark}
	\begin{theorem}\label{an+an-1 is bigger than c}
		Suppose the valuations $a_1,\dots,a_n$ of $x_1,\dots,x_n$ of $R$ satisfy $a_n+a_{n-1}\geq c_R$, then the torsion $\tau(\Omega_R)$ is nonzero.
	\end{theorem}
	\begin{proof}  Throughout the proof, we can assume that $\mathfrak{C}_R\subseteq \m_R^2$ as otherwise the conclusion follows from \cite[Theorem 3.1]{huneke2021torsion}.
		
		Consider $S=R\left[\frac{\mathfrak{C}_R}{x_1}\right] $ with conductor $\mathfrak{C}_S=(t^{c_R-a_1})\overline{R}$. If $x_n\in \mathfrak{C}_s$, then $a_n\geq c_S=c_R-a_1$ and hence we are done by \Cref{prop on two valuations}(1). 

		So it is enough to assume $x_n\not\in\mathfrak{C}_S$. Since $a_{n-1}+a_n\geq c_S+a_1=(c_R-a_1)+a_1$, we use the same steps as in \Cref{prop on two valuations}(2) and end up at the torsion element  \eqref{section 6 eq3}, now in $\Omega_S$ (we put zeros in the last $s$ rows corresponding to $dT_1,\dots, dT_s$). This torsion element is nonzero as it is nonzero in $\Omega_{S/(\mathfrak{C}_S+\m_S^2)}$.

		Let $\tau$ denote this torsion element. 
		Next we show that $\tau, \mathfrak{r}_{ij}=[x_i,T_j], {\Gamma_{ij}=[T_i,T_j]}$ (as in \Cref{explicit Sec}) are $\k$-linearly independent. In the description of $\tau$, let $u=f_u(x_1,\dots,x_n)+\sum\limits_{k=1}^s\gamma_k T_k, \gamma_k\in\k$. Since $x_n\not\in \mathfrak{C}_S$ and $u\in \mathfrak{C}_S$, we have $f_u(x_1,\dots,x_n)\in\m_R^2$. Thus $\tau$ in $\Omega_{S/\m_S^2}$, will take the form
		\begin{align*}
			\tau=\begin{bmatrix}
				\sum\limits_{k=1}^s\gamma_k T_k\\
				0\\\vdots\\0\\ a_1x_{n-1}\\0\\\vdots\\0
			\end{bmatrix}\hspace{-1em}
			\begin{array}{c;{2pt/2pt}c}
				~&~\\ ~&~\\ ~&~\\ ~&~\\ ~&~\\ ~&~\\ ~&~\\~&~\\~&~\\~&~\\
			\end{array}\hspace{-1em}
			\overbrace{
				\begin{array}{c}
					~\\
					dx_1 \\dx_2\\ \vdots\\dx_{n-1} \\ dx_n \\dT_1\\ \vdots \\ dT_s
				\end{array}
			}^{\text{basis}}.
		\end{align*}
		Let  $J=\m_S^2+\langle T_1,\dots,T_s\rangle$. In $\Omega_{S/J}$, the above torsion takes the from $a_1x_{n-1}dx_n$ which is nonzero by \cite[Proposition 2.6]{ABC1}. Now suppose $\tau=\sum_{i,j} k_{ij}[x_i,T_j]{+\sum_{i,j}k'_{ij}\Gamma_{ij}}$, $k_{ij},k'_{ij}\in\k$. In $\Omega_{S/J}$, $[x_i,T_j]={\Gamma_{ij}=0}$ (see \eqref{simpler description of [x_i,T_j]}) and this contradicts the fact that $\tau$ is nonzero in $\Omega_{S/J}$. Thus $\tau$ is $\k$-linearly independent with $[x_i,T_j],\Gamma_{ij}$. Now the result follows from \cite[Theorem 4.9]{huneke2021torsion}.
			\end{proof}

\begin{proposition}\label{monomialcase}\rm{(\textbf{Monomial Case})}\label{monomial case}
	Let $K=(m_1,\dots,m_s)$ be a monomial ideal minimally generated by monomials $m_i=X_1^{b_{i1}}\cdots X_n^{b_{in}}$ in $\k\ps{X_1,\dots,X_n}$ and $R'=\k\ps{X_1,\dots,X_n}/K$. Suppose $m=X_1^{b_1}\cdots X_n^{b_n}\not\in K$ where $b_u\geq 1$ for some $1\leq u\leq n$. Then $m d x_u=0$ in $\Omega_{R'}$ if and only if $X_u^{b_u+1}\in K$ or  $\frac{\partial(mX_u)}{\partial X_i}\in K$ for all $1\leq i\neq u\leq n$. Here $x_i$ denotes the image of $X_i$ in $R'$.
\end{proposition}
\begin{proof}
	First, let $mdx_u=0$. Without loss of generality we assume that $u=1,i=2$.  Suppose by contradiction, $X_1^{b_1+1}\not\in K$ and  $\frac{mX_1}{X_2}\not \in K$. Here we replaced $\frac{\partial(mX_u)}{\partial X_i}=\frac{\partial(mX_1)}{\partial X_2}$ by $\frac{mX_1}{X_2}$.

	Consider the monomial ideal $L=(X_1^{b_1+1}X_2^{b_2})+K=(X_1^{b_1+1}X_2^{b_2},m_1,\dots m_s)$. First suppose that $m_i$, for some $1\leq i\leq s$, divides $X_1^{b_1+1}X_2^{b_2}$. Then $m_i=X_1^uX_2^v$ where $u\leq b_1+1,v\leq b_2$. If $u<b_1+1,v\leq b_2$, then $m\in (m_i)\subseteq K$, a contradiction. Now suppose $u=b_1+1,v<b_2$, then $\frac{mX_1}{X_2}\in (m_i)\subseteq K$, a contradiction.  Therefore, in these cases, none of the $m_i$'s divide $X_1^{b_1+1}X_2^{b_2}$. Finally if $u=b_1+1,v=b_2$, then $m_i=X_1^{b_1+1}X_2^{b_2}$ and hence $K=L$. Let $(X_1^{b_1+1}X_2^{b_2},m_1,\dots m_l), l\leq s$ be a minimal generating set for $L$. Also, clearly  $\frac{mX_1}{X_2}=X_1^{b_1+1}X_2^{b_2-1}X_2^{b_3}\cdots X_n^{b_n}\not\in L$. 

	Now since $mdx_1=0$ in $\Omega_{R'}$, we also have $mdx_1=0$ in $\Omega_{\k[[X_1,\dots,X_n]]/L}$. Thus
	\begin{align*}
		mdx_1=r d(x_1^{b_1+1}x_2^{b_2})+\sum_{i=1}^l r_idm_i, r,r_i\in {\k[[X_1,\dots,X_n]]/L}
	\end{align*}
	{(here we have again used $x_i$ to denote the images in the suitable quotient)}. Therefore,
	\begin{align}
		m&=r\frac{\pp(x_1^{b_1+1}x_2^{b_2})}{\pp x_1}+\sum_{i=1}^l r_i \frac{\pp m_i}{\pp x_1} \label{eq1}
		=r(b_1+1)x_1^{b_1}x_2^{b_2}+\sum_{i=1}^l r_ib_{i1}x_1^{b_{i1}-1}x_2^{b_{i2}}\cdots x_n^{b_{in}}.\\
		0&=r\frac{\pp(x_1^{b_1+1}x_2^{b_2})}{\pp x_k}+\sum_{i=1}^lr_i \frac{\pp m_i}{\pp x_k} \text{ for } 2\leq k\leq n.\label{eq2}
	\end{align}
	First we analyze \eqref{eq1}. Suppose $\frac{\pp m_i}{\pp x_1}$ divides $m$, then we have $b_{i1}-1\leq b_{1}$ and $b_{ik}\leq b_k$ for $2\leq k\leq n$. If $b_{i1}-1<b_1$, then $m\in K$ which leads to a contradiction. Therefore, $b_{i1}=b_1+1$. Now if $b_{i2}<b_2$, then $\frac{mX_1}{X_2}\in K$ which is again a contradiction. Thus $b_{i2}=b_2$. But now this forces $m_i\in (X_1^{b_1+1}X_2^{b_2})$ which is also contradiction as $(X_1^{b_1+1}X_2^{b_2},m_1,\dots, m_l)$ is a minimal generating set for $L$. Thus $\frac{\pp m_i}{\pp x_1}$ does not divide $m$. This shows that $m_i$'s have the property that either $b_{i1}>b_1+1$ or $b_{ik}>b_k$ for some $2\leq k\leq n$. 

	From the above discussion and \eqref{eq1}, we have $r=\frac{1}{b_1+1}x_3^{b_3}\cdots x_n^{b_n}$. Using $k=2$ in \eqref{eq2}, we have 
	\begin{align*}
		0&=r\frac{\partial(x_1^{b_1+1}x_2^{b_2})}{\partial x_2}+\sum_{i=1}^s r_i\frac{\partial m_i}{\partial x_2} \\
		&=rb_2x_1^{b_1+1}x_2^{b_2-1}+\sum_{i=1}^sr_i b_{i2}x_1^{b_{i1}}x_2^{b_{i2}-1}\cdots x_n^{b_{in}} \\
		&=\frac{b_2}{b_1+1}\frac{mx_1}{x_2}+\sum_{i=1}^sr_i b_{i2}x_1^{b_{i1}}x_2^{b_{i2}-1}\cdots x_n^{b_{in}}
	\end{align*}
	By our choice of $m_i$ (preceeding discussion on choice of $b_{ik}$), we have that the term $\frac{b_2}{b_1+1}\frac{mx_1}{x_2}$ cannot cancel with any term in the sum $\sum_{i=1}^sr_i b_{i2}x_1^{b_{i1}}x_2^{b_{i2}-1}\cdots x_n^{b_{in}}$. Thus $\frac{b_2}{b_1+1}\frac{mx_1}{x_2}=0$ or in other words, $\frac{b_2}{b_1+1}\frac{mX_1}{X_2}\in L$. Since $\frac{b_2}{b_1+1}\in\k$, this forces $\frac{mX_1}{X_2}\in L$, which leads to a contradiction. Thus we have either $X_1^{b_u+1}\in K$ or $\frac{\partial (mX_u)}{\partial X_i}\in K$.

	Conversely, suppose $\frac{\partial(mX_u)}{\partial X_i}\in K$ for all $1\leq i\neq u\leq n$. This implies $mX_u\in K$ and hence $d(mX_u)=0$ in $\Omega_{R'}$. Hence $\ds (b_u+1)mdx_u+\sum_{i=1,i\neq u}^nb_i\frac{mx_u}{x_i}dx_i\in KdR'+dK$. Using the assumption, we have $(b_u+1)mdx_u\in KdR'+dK$. Also, clearly,  if $X_u^{b_u+1}\in K$, then $mdx_u=0$, completing the proof. 
\end{proof}
In the following, we {use the notion of} monomial support of an ideal which is the same as $\ms(I)$ in \cite{PoliniUlrichVitulli07}. The ideal $\ms(I)$ is defined to be the smallest monomial ideal containing $I$ and is generated by the monomial support of a generating set for $I$.

\begin{example}
	If $I=(X^3-YZ,X^2Y-Z^2,XZ-Y^2)\subseteq S=\k\ps{X,Y,Z}$, then $\ms(I)=(X^3,YZ,X^2Y,Z^2,XZ,Y^2)$. Notice that $I$ is the defining ideal of $\k\ps{t^3,t^4,t^5}$. Clearly, we have $\tau=3xdz-5zdx\in \tau(\Omega_{S/I})$.  We can now use \Cref{monomialcase} with $K=\ms(I)$ and $mdx_u=xdz$ to see that $\tau$ is nonzero in $\Omega_{S/\ms(I)}$. Hence, it is nonzero in $\Omega_{S/I}$.
\end{example}

\begin{theorem}\label{m^n in C}
	Let $R=\k\ps{\alpha_1t^{a_1},\dots,\alpha_nt^{a_n}}\cong \k\ps{X_1,\dots,X_n}/I$ with conductor $\mathfrak{C}_R=(t^{c_R})\overline{R}$. Let $\tau(\Omega_R)$ denote the torsion submodule of $\Omega_R$ and $x_i$ denote the images of $X_i$ in $R$. {Let $N\geq 2$ be the  integer such that} $\m^N\subseteq \mathfrak{C}_R$ and suppose that $x_1^{N-1},x_1^{N-2}x_2\not\in\ms(I)$ where $I$ is the defining ideal of $R$. Then the torsion submodule $\tau(\Omega_R)$ is nonzero.
\end{theorem}
We can always choose $N$ to be the least integer satisfying all the conditions in the hypothesis. For example, if $k$ is the least integer such that $\m^k\subseteq \C_R$ and $k\leq N$, then $x_1^{k-1},x_1^{k-2}x_2\not\in\ms(I)$ (as  $x_1^{n-1},x_1^{N-2}x_2\not\in\ms(I)$) as well. Thus, we will assume $N$ is least henceforth. The hypothesis now guarantees that $x_1^{N-1}\not\in\C_R$. Before we go to the proof of the theorem, we present a small lemma which helps us reduce to the case that $x_1^{N-2}x_2\not\in\C_R$.
\begin{lemma}\label{m^N in C lemma}
	Under the conditions of the above theorem, if $x_1^{N-2}x_2\in\C_R$, then the torsion $\tau(\Omega_R)$ is nonzero.
\end{lemma}
\begin{proof}
	Since $x_1^{N-2}x_2\in\C_R$, we have $t^{(N-2)a_1+a_2}\in \C_R$. Thus there exists $f(\underline{x})$ such that $t^{(N-2)a_1+a_2}=x_1^{N-2}x_2-f(\underline{x})$. Since the valuation $v(x_1^{N-2}x_2)=a_1(N-2)+a_2$,  the valuation of $f(\underline{x})$ is more than $(N-2)a_1+a_2$. Now we know that $a_1(t^{a_1})dt^{(N-2)a_1+a_2}-((N-2)a_1+a_2)t^{(N-2)a_1+a_2}dt^{a_1}=0$. Thus we have the torsion element
	\begin{multline*}
		 a_1x_1d(x_1^{N-2}x_2-f(\underline{x}))-\left((N-2)a_1+a_2)x_1^{N-2}x_2-f(\underline{x})\right)dx_1\\
		=a_1x_1\left((N-2)x_1^{N-3}x_2dx_1+x_1^{N-2}dx_2-\sum_i\frac{\partial f}{\partial x_i}dx_i \right)-\left((N-2)a_1+a_2)x_1^{N-2}x_2-f(\underline{x})\right)dx_1\\
		=(N-2)a_1x_1^{N-2}x_2dx_1+a_1x_1^{N-1}dx_2-\sum_ia_1x_1\frac{\partial f}{\partial x_i}dx_i-\left((N-2)a_1+a_2)x_1^{N-2}x_2-f(\underline{x})\right)dx_1\\
		\hfilneg  =\left(-a_2x_1^{N-2}x_2+f(\underline{x})\right)dx_1+a_1x_1^{N-1}dx_2-\sum_ia_1x_1\frac{\partial f}{\partial x_i}dx_i.	\hspace{10000pt minus 1fil}
	\end{multline*}
Let $J=\langle x_3,\dots,x_n\rangle+\ms\langle p\in R~|~v(p)>(N-2)a_1+a_2\rangle+\ms(I)$ (we treat the last part as an ideal in $R$; we keep writing $\ms(I)$ for notational convenience) where $I$ is the defining ideal of $R$. Note that $J$ is a monomial ideal not containing $x_1^{N-1},x_1^{N-2}x_2$ as these are not in $ \ms(I)$  and its valuations are less than or equal to $(N-2)a_1+a_2$. Also notice that $f(\underline{x})\in J$ as the valuation of $f(\underline{x})$ is more than $(N-2)a_1+a_2$. Now in $\Omega_{R/J}$,  {we show that} the above torsion element takes the form $	\left(-a_2x_1^{N-2}x_2\right)dx_1+a_1x_1^{N-1}dx_2$.  {First, } $dx_i=0$ for $i\geq 3$ and hence it is enough to show that  $\sum_{i=1}^2a_1x_1\frac{\partial f}{\partial x_i}dx_i=0$ in $\Omega_{R/J}$. 
	

Since $f(\underline{x})$ has valuation more than $(N-2)a_1+a_2$, $a_1x_1\frac{\partial f}{\partial x_1}$ is either $0$ or has valuation more than $a_1+(N-3)a_1+a_2=(N-2)a_1+a_2$. Thus it vanishes in $R/J$. So all that remains is $a_1x_1\frac{\partial f}{\partial x_2}$ which is either $0$ or has valuation more than $a_1+(N-2)a_1=(N-1)a_1=v(x_1^{N-1})$. Since $x_1^{N-1}\not \in \ms(I)$ and  $x_i\in J$ for $i\geq 3$, we need to only focus on the case when one of the  terms in the product $a_1x_1\frac{\partial f}{\partial x_2}$ is possibly $x_1^{N-2}x_2$.  This in turn shows that $f(\underline{x})$ has a term $\beta x_1^{N-3}x_2^2, \beta\in \k$. But the valuation shows that $x_1^{N-3}x_2^2\in J$. Hence $a_1x_1\frac{\partial f}{\partial x_1}dx_1+a_1x_1\frac{\partial f}{\partial x_2}dx_2=0+\beta a_1x_1d(x_1^{N-3}x_2^2)=0$ in $\Omega_{R/J}$, i.e., $\sum_ia_1x_1\frac{\partial f}{\partial x_i}dx_i=0$ in $\Omega_{R/J}$ as well. 

Now in $\Omega_{R/J}$, the torsion element $	\left(-a_2x_1^{N-2}x_2\right)dx_1+a_1x_1^{N-1}dx_2$ simplifies to $\omega=-(a_2+a_1(N-1))x_1^{N-2}x_2dx_1$. This is because $x_1^{N-1}x_2\in J$ and hence $x_1^{N-1}dx_2+(N-1)x_1^{N-2}x_2dx_1=0$ in $\Omega_{R/J}$. Now \Cref{monomial case} can be applied with $m=x_1^{N-2}x_2$ to see that $\omega$ is nonzero.
\end{proof}
\begin{proof}[Proof of \Cref{m^n in C}]  Due to the previous lemma, we assume $x_1^{N-1},x_1^{N-2}x_2\not\in\C_R$. Also, we can safely assume that we have ``monomialized'' $x_1$ so that $R=k\ps{t^{a_1},\alpha_2t^{a_2},\dots,\alpha_nt^{a_n}}$.

	Consider
	\begin{align*}
		\frac{dx_2}{dt}\left(	\frac{dx_1}{dt}\right)-	\frac{dx_1}{dt}\left(	\frac{dx_2}{dt}\right)=0\\
		(\alpha_2a_2t^{a_2-1}+t^{a_2}\frac{\partial \alpha_2}{\partial t})\left(	\frac{dx_1}{dt}\right)-a_1t^{a_1-1}\left(\frac{dx_2}{dt}\right)=0
	\end{align*}
	Multiply both sides by $t^{(N-2)a_1+1}$ to get
	\begin{align}
		(\alpha_2a_2t^{(N-2)a_1+a_2}+t^{(N-2)a_1+a_2+1}\frac{\partial \alpha_2}{\partial t})\left(	\frac{dx_1}{dt}\right)-a_1t^{(N-1)a_1}\left(\frac{dx_2}{dt}\right)=0\nonumber\\
		(a_2x_2x_1^{N-2}+t^{(N-2)a_1+a_2+1}\frac{\partial \alpha_2}{\partial t})\left(	\frac{dx_1}{dt}\right)-a_1x_1^{N-1}\left(\frac{dx_2}{dt}\right)=0\label{monomial theorem main eq1}
	\end{align}
{We } have $S=R[T_1,\dots,T_s]$ and $\mathfrak{C}_S=(t^{c-a_1})\overline{R}$. Since $(N-2)a_1+a_2\geq c-a_1$ ($\m^{N}\subseteq \mathfrak{C}_R$), we have that the above torsion element is a member of $\tau(\Omega_S)$.

	The above torsion element takes the form:
	\begin{align*}
		\tau=\begin{bmatrix}
			a_2x_2x_1^{N-2}+t^{(N-2)a_1+a_2+1}\frac{\partial \alpha_2}{\partial t}\\
			-a_1x_1^{N-1}\\
			0\\\vdots\\0
		\end{bmatrix}.
	\end{align*}
	For the defining ideal $L$ of $S$, let the monomial support be denoted by $\ms(L)$.

	We first show that $\tau$ is nonzero in $\Omega_S$. Let $J'=\langle x_3,\dots,x_n,T_1,\dots,T_s\rangle+\ms(L)+\ms(p\in S~|~v(p)>(N-2)a_1+a_2)$. It is a monomial ideal. Now consider $\tau$ in $\Omega_{S/J'}$. Since the valuation of $t^{(N-2)a_1+a_2+1}\frac{\partial \alpha_2}{\partial t}$ is more than $(N-2)a_1+a_2$, it will be of the form 
	\begin{align*}
		\overline{\tau}=\begin{bmatrix}
			a_2x_1^{N-2}x_2\\
			-a_1x_1^{N-1}\\
			0\\\vdots\\0
		\end{bmatrix}\in \Omega_{S/J'}.
	\end{align*}
	By hypothesis, $x_1^{N-2}x_2,x_1^{N-1}$ are not in the monomial support of the defining ideal of $R$. Now the rest of the defining ideal of $S$ is of the form $x_iT_j-g_{ij}(\underline{x}), T_iT_j-h_{ij}(\underline{x})$. The valuation of $x_i T_j$ is $a_i+b$ where $b\geq c_R-a_1$. Since $x_1^{N-1}\not\in \mathfrak{C}_R$, the valuation $c_R>(N-1)a_1$. Thus for $i\geq 2$, we have {$a_i+b\geq a_i+c_R-a_1>a_i+(N-2)a_1\geq a_2+(N-2)a_1=v(x_1^{N-2}x_2)>v(x_1^{N-1})$}. Thus the monomial support of $g_{ij}(\underline{x}), i\geq 2$ can never have $x_1^{N-2}x_2,x_1^{N-1}$ as terms. Also, the valuation of $x_1T_j$ is at least $c_R>(N-2)a_1+a_2$ (as $x_1^{N-2}x_2\not\in\C_R$). Thus $g_{1j}(\underline{x})$ cannot have $x_1^{N-2}x_2,x_1^{N-1}$ as terms as well. Applying similar arguments to analyze the elements $T_iT_j-h_{ij}(\underline{x})$, we conclude that $x_1^{N-2}x_2,x_1^{N-1}$ cannot appear in the monomial support of $L$. Thus these terms do not lie in $J'$ too.

	Since $x_1^{N-1}x_2\in J'$, we have $d(x_1^{N-1}x_2)=0$ in $\Omega_{S/J'}$. Thus we have $(N-1)x_1^{N-2}x_2dx_1=-x_1^{N-1}dx_2$. Using this equality in the above torsion element we have
	\begin{align}\label{monomialeq1}
		\overline{\tau}&=(a_2x_1^{N-2}x_2)dx_1-a_1x_1^{N-1}dx_2\nonumber\\
		&=(a_2x_1^{N-2}x_2)dx_1+a_1(N-1)x_1^{N-2}x_2dx_1\nonumber\\
		&=(a_2+a_1(N-1))x_1^{N-2}x_2dx_1.
	\end{align}
	Clearly $(a_2+a_1(N-1))\neq 0$, and $x_1^{N-2}x_2\not\in J'$. Using \Cref{monomial case}, we get that the above torsion element in nonzero in $\Omega_{S/J'}$ and hence $\tau$ is nonzero in $\Omega_{S}$.

	Next we show that $\tau$ is $\k$-linearly independent with the torsion elements $[x_i,T_j],{\Gamma_{ij}}$ (using the notations as in \Cref{explicit Sec}). Suppose \begin{align}\label{neweq}
	\tau=\sum_{i,j}k_{ij}[x_i,T_j]+\sum_{i,j}k'_{ij}\Gamma_{ij}
	\end{align} where $k_{ij},k'_{ij}\in\k$. 

	Let $J''=\langle x_1\rangle+\m_S^2$. Then in $\Omega_{S/J''}$, the torsion element $\tau $ \eqref{monomial theorem main eq1} is zero. Since $[x_1,T_j]=a_1x_1dT_j-b_jT_jdx_1$, this element is also zero in $\Omega_{S/J''}$. Thus in $\Omega_{S/J''}$, \Cref{neweq} is of the form $0=\sum_{i\geq 2,j\geq 1}k_{ij}\overline{[x_i,T_j]}+\sum_{i,j}k'_{ij}\overline{\Gamma_{ij}}$. Now consider the description of $\overline{[x_i,T_j]},i\geq 2,j\geq 1$  in \eqref{description of [x_i,T_j] with an bigger than c-a1}. Following the same proof as in \Cref{totalcount}, it is easy to see that $\overline{[x_i,T_j]},i\geq 2,j\geq 1$ {and the $\overline{\Gamma_{ij}}$ } are still $\k$-linearly independent in $\Omega_{S/J''}$.  It follows that $k_{ij}=0$ for $i\geq 2, j\geq 1$ {and also $k'_{ij}=0$ for $1\leq i,j\leq s$}. Thus we have {$\tau=\sum_{j\geq 1}k_{ij}[x_1,T_j]$ } in $\Omega_S$. Now consider the ideal $J'=\langle x_3,\dots,x_n,T_1,\dots,T_s\rangle+\ms(L)+\ms(p\in S~|~v(p)>(N-2)a_1+a_2)$ from the previous part of the proof. In $\Omega_{S/J'}$, $\tau$ is nonzero which we already proved above. But $\overline{[x_1,T_j]}=0$ in $\Omega_{S/J'}$ as $T_j\in J'$. Thus we arrive at $\tau=0$ in $\Omega_{S/J'}$, a contradiction. This proves the $\k$-linear independence of $\tau,[x_i,T_j],\Gamma_{ij}$. {Thus, $\lambda(\tau(\Omega_S))\geq ns+{s\choose 2}+1$.
	}

	Hence, a $\k$-linear combination of $\tau,[x_i,T_j],\Gamma_{ij}$ can be pulled back to a nonzero torsion element in $\Omega_R$ by \cite[Theorem 4.9]{huneke2021torsion}.
\end{proof}
The above theorem recovers \cite[Theorem 2.13]{ABC1}. 
\begin{corollary}{\cite[Theorem 2.13]{ABC1}} Under the conditions of \Cref{m^n in C}, if $\m^3\subseteq \C_R$, then the torsion $\tau(\Omega_R)$ is nonzero.
\end{corollary}
\begin{proof}
Since $\m^3\subseteq \C_R$, we only need to check if $x_1^2,x_1x_2\not\in\ms(I)$ to use the above theorem. This is equivalent to checking if $X_1^2-f(\underline{X}),X_1X_2-g(\underline{X})\in I$ or $x_1^2=f(\underline{x}),x_1x_2=g(\underline{x})$ in $R$. 


We have $v(x_1^2)=2a_1,v(x_1x_2)=a_1+a_2$. Also, note that $v(f)$ or $v(g)$ is of the form $\sum_{i=1}^n j_ia_i, j_i\in\mathbb{N}$. Since $I\subseteq \m^2$, $f(\underline{X})$ and $g(\underline{X})$ cannot have any linear terms in $X_i$. 
Thus, either $j_i\geq 2$ for some $i\geq 2$ or there exist at least two values of $i$, say $\{i_1,i_2\}\neq \{1,2\}$, such that $j_i\geq 1$. Since we have $a_1<a_2<\cdots<a_n$, it is clear that the valuations $2a_1,a_1+a_2$ cannot be attained in either of these situations. 
Hence, $x_1^2,x_1x_2\not\in\ms(I)$.
\end{proof}

\begin{example}
	Let $R=\k\ps{t^{22},t^{23}+t^{27},t^{24}+t^{27},t^{25}+t^{27},t^{26}+t^{27}}\cong \k\ps{x,y,z,w,u}$. Here $\mathfrak{C}_R=(t^{110})\overline{R}$ and hence $\m^5\subseteq \mathfrak{C}_R$. M2 computations show that  $x^4,x^3y\not\in\ms(I)$. Thus the torsion submodule $\tau(\Omega_R)\neq 0$ by the above theorem.
\end{example}
\begin{example}
	Let $R=\k\ps{t^{30},t^{31}+t^{36},t^{32}+t^{36},t^{33}+t^{36},t^{34}+t^{36}}\cong \k\ps{x,y,z,w,u}$. Here $\mathfrak{C}_R=(t^{180})\overline{R}$ and hence $\m^6\subseteq \mathfrak{C}_R$. M2 computations show that   $x^5,x^4y\not\in\ms(I)$. Thus the torsion submodule $\tau(\Omega_R)\neq 0$ by the above theorem.
\end{example}

\end{document}